\documentclass[11pt]{amsart}
\usepackage{amsmath,amsthm,amssymb}
\usepackage{geometry}                
\usepackage{graphicx}
\usepackage{amssymb}
\usepackage{epstopdf}
\theoremstyle{plain}
\newtheorem{thm}{Theorem}

\theoremstyle{plain}
\newtheorem{prop}{Proposition}

\theoremstyle{plain}
\newtheorem{lem}{Lemma}

\theoremstyle{plain}
\newtheorem{cor}{Corollary}

\theoremstyle{definition}
\newtheorem{dfn}{Definition}

\title{A Uniform Bound for the Order of Monodromy}
\author{Naoya Umezaki}
\address{Graduate School of Mathematical Sciences, The University of Tokyo, Tokyo 153-8914, Japan}
\email{umezaki@ms.u-tokyo.ac.jp}

\begin{document}
\maketitle

\begin{abstract}
We give a bound for the order of the local monodromy of a compatible system of l-adic representations, which is independent of l.
For the etale cohomology of a variety, the bound depends only on some numerical invariants of varieties.
\end{abstract}

\maketitle

\section{Introduction}

Let $K$ be a complete discrete valuation field, $k$ be the residue field of $K$ and $p$ be the characteristic of $k$.
Let $\ell$ be a prime number distinct from $p$.
Let $G_K$ be the absolute Galois group ${\rm Gal}(K^{sep}/K)$ and $I_K$ be the inertia group of $K$.

For an $\ell$-adic Galois representation of $G_K$,
Grothendieck's monodromy theorem is a fundamental result.

\begin{thm}[\cite{Serre_Tate:1968}, Appendix]
\label{thm2}
Let $\rho$ be an $\ell$-adic representation of $G_K$ and $k$ satisfy the following property:

$(C_\ell)$ No finite extension of $k$ contains all the roots of unity of $\ell$-power order.

Then there exists an open subgroup $H$ of $I_K$ such that $\rho(s)$ is unipotent for all $s \in H$.
\end{thm}

For a profinite group, an open subgroup has finite index.
In this paper, we discuss this index.
More precisely, if the representation appears as the $\ell$-adic etale cohomology of an algebraic variety over $K$,
we can compute a bound for the index which is uniform on $X$ and independent of $\ell$ by using some explicit numerical invariants.

The statement is the following.

\begin{thm}
\label{thm1}
Let $n$ be a positive integer, $b\in\mathbb{N}^{n}$ and $c \in \mathbb{Z}^{n-1}$.
The constant $C_{b,c,n}$ (defined in Definition 3.3) satisfies the following property:

For a smooth projective geometrically connected variety $X$ of dimension $n$ over $K$
and a very ample invertible sheaf $L$ on $X$ with $(b_i(X))_{i=1,\ldots,n} = b, (c_i(X, L))_{i=1,\ldots,n-1} = c$,
there exists an open subgroup $I$ of $I_K$ of index $[I_K:I]$ dividing $C_{b,c,n}$
such that the action of $I$ on $H^i(X_{\bar{K}}, \mathbb{Q}_\ell)$ is unipotent for every $i$ and every $\ell$.

\end{thm}

In \cite[Proposition 6.3.2]{Berthelot:1997}, 
Berthelot shows the $\ell$-independence of the subgroup by using an alteration under more general assumption
($X$ is not necessarily proper nor smooth.) 
In this paper, we show that the index is not only independent of $\ell$ but also bounded uniformly on $X$.
In some cases we show that the subgroup itself depends only on certain numerical invariants of $X$.

If the number of field extensions of $K$ whose degree is bounded is finite then the subgroup itself is independent of $X$ and $L$.
For example, $K$ is a finite extension of $\mathbb{Q}_p$.

The content of this paper is the following.
In section 2, we give a bound for the index in the case of a compatible system of $\ell$-adic representaions.
In section 3, we apply it to the case of the etale cohomology of a variety and prove the theorem.
For the proof, we reduce it to the case of a compatible system
by using induction and an $\ell$-independence result of Ochiai \cite[Theorem 2.4]{Ochiai:1999}.

In this paper, we assume that the residue field $k$ of $K$ is a perfect field and satisfies $(C_\ell)$.
We say that a representation of a topological group $G$ is an $\ell$-adic representation
if it is a continuous linear representation on a finite dimensional vector space over $\mathbb{Q}_\ell$.

\section*{Acknowledgement}
The author is grateful to his advisor Professor Takeshi Saito for helpful discussions, warm encouragement and careful reading.
He also thanks Professor Yuichiro Taguchi for suggesting the problem,
Professor Luc Illusie for his interest and comments on references
and Professor Fabrice Orgogozo for his careful reading and helpful discussion.
This work was supported by the Program for Leading Graduate Schools, MEXT, Japan.

\section{Compatible System}

In this section, we discuss a bound for a compatible system.
This will immediately imply the $n=1$ case of our theorem.

First, we define constants $C_{\ell, d}$ and $C_d$ for a non negative integer $d$ and a prime number $\ell$.
The constant $C_{\ell, d}$ denotes the order of the finite group $GL_d(\mathbb{F}_\ell)$ if $\ell$ is not equal to 2 
and $GL_d(\mathbb{Z}/4\mathbb{Z})$ if $\ell$ is equal to $2$.
The constant $C_d$ denotes $\gcd(C_{\ell, d}\mid\ell\neq p)$.
We use these constants in the definition of the constant $C_{b,c,n}$ in Definition~\ref{def1}.

\begin{prop}
Let $(\rho, V)$ be an $\ell$-adic representation of $G_K$ with $\dim V=d$.
Then the two subsets of $I_K$, $\{s \in I_K \mid \rho(s) $ is unipotent$\}$ and $\{s \in I_K \mid {\rm Tr}(\rho(s); V) = d\}$,
are equal.
This subset $I$ is an open subgroup of $I_K$ and the index $[I_K:I]$ divides $C_{\ell, d}$.
\end{prop}

This is a refinement of Theorem~\ref{thm2} proved by Grothendieck in \cite[Appendix]{Serre_Tate:1968}.
See also Deligne \cite[Th\'eor\`eme 8.2]{Deligne:1973}.

Before proving this proposition, we review the proof of Theorem~\ref{thm2}.
By taking an appropriate basis of $V$, we may assume that the image of $\rho$ is contained in a lattice $L\cong\mathbb{Z}_\ell^d$ of $V$. 
Put $I_\ell=\rho^{-1}(1+\ell M_d(\mathbb{Z}_\ell))$ (resp. $\rho^{-1}(1+4M_d(\mathbb{Z}_2))$ if $\ell=2$).
This is an open subgroup of $I_K$.

The key step is to show that if $\rho(s)$ is contained in $1+\ell^2 M_d(\mathbb{Z}_\ell)$ for $s \in I_K$, then it is unipotent.
For $\ell \neq 2$, it is enough to be contained in $1+\ell M_d(\mathbb{Z}_\ell)$.
This shows that there is an open subgroup of $I_K$ whose action on $V$ is unipotent and all the eigenvalues of $s \in I_K$ are roots of unity.

\begin{proof}
First, consider the case that the image of $\rho$ is finite.
If $X\in M_d(\mathbb{Q}_\ell)$ is unipotent and of finite order, then it is trivial.
Hence, $\ker(\rho)$ is equal to the subset of unipotent elements.
The latter contains $\ker(\rho)$ and these are equal because all the eigenvalues of $\rho(s)$ are roots of unity.

The general case is reduced to the case of a finite image as follows.

Take a continuous surjection $t:I_K \to \mathbb{Z}_\ell$ and an element $s_0 \in I_\ell$ such that $t(s_0)\neq 0$.
Then $\rho|_{I_\ell}$ factors through this $t$ and $\tilde{\rho}:\mathbb{Z}_\ell\to GL_n(\mathbb{Z}_\ell)$
because the image $\rho(I_\ell)$ is pro-$\ell$ and $\ker t$ is prime to $\ell$.
Put $N = t(s_0)^{-1} \log(\rho(s_0))$ and $r(s)=\rho(s)\exp(-t(s)N)$.
Since $\rho(s_0)\in1+\ell M_d(\mathbb{Z}_\ell)$ is unipotent, $N$ is nilpotent.

The image of $r$ is finite because $r(s)=1$ for $s\in I_\ell$.
We can show that $\rho(s)$ and $\exp(-t(s)N)$ commutes.
In fact, $\rho(ss_0s^{-1})=\tilde{\rho}(t(ss_0s^{-1}))=\tilde{\rho}(t(s_0))=\rho(s_0)$.
This implies that ${\rm Tr}(r(s))={\rm Tr}(\rho(s))$ and that $r(s)$ is unipotent if and only if $\rho(s)$ is unipotent.

Finally, we show the assertion about the index.
In fact, $I$ contains $I_\ell$ and so $[I_K: I]$ devides $[I_K:I_\ell]$.
For $\ell\neq2$, the subgroup $1+\ell M_d(\mathbb{Z}_\ell)$ is the kernel of the natural map $GL_d(\mathbb{Z}_\ell) \to GL_d(\mathbb{F}_\ell)$.
This implies that the induced map $I_K/I_\ell \to GL_d(\mathbb{F}_\ell)$ is injective and the index $[I_K:I_\ell]$ divides $C_{\ell, d}$.
The same argument works for $\ell=2$.
\end{proof}

The smooth representation $r$ and the nilpotent operator $N$ defined in the above proof is the Weil-Deligne representation associated to $\rho$.
See \cite[D\'efinition 8.4.1]{Deligne:1973}.

As a corollary of this proposition, we get the following, which is a key step of the proof of the theorem.

\begin{cor}
\label{cor1}
Let $\{(\rho_\ell, V_\ell)\}_{\ell \neq p}$ be a family of $\ell$-adic representations of $G_K$
such that ${\rm Tr}(\rho_\ell(s);V_\ell) = {\rm Tr}(\rho_{\ell'}(s);V_{\ell'}) \in \mathbb{Q}$ for every $\ell$, $\ell'$ and $s \in I_K$.
Then the subgroup $I = \{s \in I_K \mid \rho_\ell(s) $ is unipotent$\}$ is independent of $\ell$
and the index $[I_K:I]$ divides $C_d$.
\end{cor}

We can improve this bound in some cases as follows.

Let $\{(\rho_\ell, V_\ell)\}_{\ell\neq p}$ be a family of $\ell$-adic representations of $G_K$.
Suppose the image of the inertia group $I_K$ to be topologically generated by one element whose trace is an $\ell$-independent rational number.
Then the characteristic polynomial of a generator has rational coefficients.
If all the eigenvalues are primitive $i$-th roots of unity,
then the degree of the characteristic polynomial is $\varphi(i)=\sharp(\mathbb{Z}/i\mathbb{Z})^\times$.
Then we get a bound $\varphi(i)\leq\dim V_\ell$, which is independent of $\ell$.

In general, the inertia group is an extension of a group topologically generated by one element by a pro-$p$ group $P_K$.
The order of $\rho(P_K)$ is bounded by the $p$-part of $C_d$ for $d=\dim V_\ell$.
In particular, if $p<2d+1$ then the action of $P_K$ is trivial.
The order of $I_K/P_K$ is bounded by Euler function $\varphi$ in the same way discussed above.

If $K=\mathbb{C}((t))$, we can not apply Grothendieck's theorem.
But if the representation is the etale cohomology of some variety over $K$,
then the eigenvalues of the monodromy action are roots of unity by the comparison theorem and the monodromy theorem for the singular cohomology.
In this case, we can apply the bound using the Euler function.

We note that these bounds depend only on the dimension of the representation.

\section{Theorem}

In this section, we prove the theorem.
First, we introduce some numerical invariants to compute the bound.

\begin{dfn}
\label{def1}
Let $X$ be a smooth projective variety over a separably closed field $k$ of dimension $n$,
$L$ a very ample invertible sheaf on $X$ and $\ell\neq char(k)$ a prime number.
We define numerical invariants $b_i(X)$ and $c_i(X, L)$ as follows.

\begin{gather*}
b_i(X) = \dim H^i(X, \mathbb{Q}_\ell)\\
c_i(X, L) = f_*(c(\check{\Omega}^1_X)c(L)^{-i}c_1(L)^i)).
\end{gather*}

Here,  $c_i(E) \in H^{2i}(X, \mathbb{Q}_\ell(i)))$ is the i-th Chern class and
$c(E) \in \bigoplus H^{2i}(X, \mathbb{Q}_\ell(i))$ is the total Chern class of a locally free sheaf $E$.
The map $f$ is the structure map of $X$ and $f_*$ is the Gysin morphism.
\end{dfn}

These numbers are integers and independent of $\ell$.
(See \cite[1.4]{Illusie:2006}. For $\ell$-independence of $c_i(X, L)$, we use Lemma~\ref{lem1}.)

For a variety $X$ over $k$ and a very ample invertible sheaf $L$ on $X$, a $j$-hyperplane section means
an iterated hyperplane section $X_j=X\cap H_1\cap \cdots \cap H_j$ of $X$ of codimension $j$, here $H_i$ is a hyperplane defined by $L$.

\begin{lem}[\cite{Grothendieck:1977}, Expos\'e VIII, Proposition 7.3]
\label{lem1}
Let $X$ be a smooth projective connected scheme over a separably closed field of dimension $n$
and $L$ a very ample invertible sheaf on $X$.
Take a smooth $j$-hyperplane section $X_j$ of $X$ defined by $L$.
Then
\begin{gather*}
c_j(X, L) = \sum_i (-1)^ib_i(X_j)\\
b_{n-j}(X_j) = (-1)^{n-j}\left(c_j(X, L) - 2\sum_{i=0}^{i=n-j-1}(-1)^i b_i(X)\right)
\end{gather*}
and it does not depend on the choice of hyperplane section.
\end{lem}

\begin{dfn}
Let $n$ be a positive integer,
$b=(b_1, \cdots, b_n) \in \mathbb{N}^{n}$ and $c=(c_1, \cdots, c_{n-1}) \in \mathbb{Z}^{n-1}$.
We define $d_{b,c} = (d_{b,c,1}, \cdots,  d_{b,c,n}) \in \mathbb{N}^{n}$
and the constant $C_{b,c,h}$ for a positive integer $h\leq n$ as follows.
\begin{gather*}
d_{b,c,j} = 
\begin{cases}
(-1)^j\left(c_{n-j} - 2\sum_{i=0}^{i=j-1}(-1)^i b_i\right) & j \neq n \\
b_j & j =n \\
\end{cases}
\\
C_{b,c,h} = \prod_{j=1}^{j=h} C_{d_{b,c,j}}.
\end{gather*}
Here $C_{d_{b,c,j}}$ is the constant defined in the beginning of the previous section.
\end{dfn}

If $X$  and $L$ satisfies $b_i(X) = b_i$ and $c_i(X,L)=c_i$ for all $i$, $d_{b,c,j}$ is the middle Betti number $b_j(X_{n-j})$ of a smooth $(n-j)$-hyperplane section $X_{n-j}$ of $X$.

Now we state and prove the theorem.
\begin{thm}[Theorem~\ref{thm1}]
Let $n$ be a positive integer, $b\in\mathbb{N}^{n}$ and $c \in \mathbb{Z}^{n-1}$.
The constant $C_{b,c,n}$ satisfies the following property:

For a smooth projective geometrically connected variety $X$ of dimension $n$ over $K$
and a very ample invertible sheaf $L$ on $X$ with $(b_i(X))_{i=1,\ldots,n} = b, (c_i(X, L))_{i=1,\ldots,n-1} = c$,
there exists an open subgroup $I$ of $I_K$ of index $[I_K:I]$ dividing $C_{b,c,n}$
such that the action of $I$ on $H^i(X_{\bar{K}}, \mathbb{Q}_\ell)$ is unipotent for every $i$ and every $\ell\neq p$.

\end{thm}

\begin{proof}
We prove this theorem by induction on $n$.

Case $n=1$.
In this case, the theorem is essentially proved in section 2.
If $X$ is a curve, $I_K$ acts trivially on $H^0(X_{\bar{K}}, \mathbb{Q}_\ell) \cong \mathbb{Q}_\ell$
and $H^2(X_{\bar{K}}, \mathbb{Q}_\ell) \cong \mathbb{Q}_\ell(-1)$.
The $\ell$-independence of the trace of $H^1(X_{\bar{K}}, \mathbb{Q}_\ell)$ is proved in \cite[Expos\'e IX, Theorem 4.3]{Grothendieck:1972}.
By applying Corollary~\ref{cor1} to $H^1(X_{\bar{K}}, \mathbb{Q}_\ell)$, we get the result.

Case general $n$.
Assume that for all $k<n$, $b\in\mathbb{N}^k$ and $c\in\mathbb{Z}^{k-1}$, the constant $C_{b,c,k}$ satisfies the property of the theorem.
Take $b\in\mathbb{N}^n$ and $c\in\mathbb{Z}^{n-1}$.
Let $X$ be a variety and $L$ a very ample invertible sheaf satisfying the assumption of the theorem.

Step 1.
First, we show that there exists an open subgroup $J$ such that the action on $H^i(X_{\bar{K}}, \mathbb{Q}_\ell)$ for $i\neq n$ is unipotent
and the index divides $C_{b,c,n-1}$.

By the Bertini theorem, there exists a smooth hyperplane section of $X$ defined by a section of $L$.
Take such a hyperplane section $Y$.
Then $L|_Y$ is a very ample invertible sheaf on $Y$.
Put $\tilde{b} = (b_i(Y))_{i=1,\ldots, n-1}$ and $\tilde{c}=(c_i(Y, L|_Y))_{i=1,\ldots, n-2}$.
By the induction assumption, there exists an open subgroup $J$ of $I_K$ whose index divides $C_{\tilde{b}, \tilde{c}, n-1}$
and the action on $H^i(Y_{\bar{K}}, \mathbb{Q}_\ell)$ is unipotent for all $i$ and $\ell$.

We claim that $C_{\tilde{b}, \tilde{c},n-1} = C_{b, c,n-1}$.
By Lemma~\ref{lem1}, $d_{b,c,j}$ is the middle Betti number of an $(n-j)$-hyperplane section of $X$
and $d_{\tilde{b},\tilde{c},j}$ is of an $(n-j-1)$-hyperplane section of $Y$.
Because this is independent of the choice of a hyperplane section, $d_{b,c,j}=d_{\tilde{b},\tilde{c},j}$ and the claim holds.

By the weak Lefschetz theorem,
the restriction map 
$H^i(X_{\bar{K}}, \mathbb{Q}_\ell) \to H^i(Y_{\bar{K}}, \mathbb{Q}_\ell)$
is injective for $i\leq n-1$.
Hence the action of $J$ on $H^i(X_{\bar{K}}, \mathbb{Q}_\ell)$ is unipotent for $i\leq n-1$.
By the Poincare duality, the action of $J$ on $H^i(X_{\bar{K}}, \mathbb{Q}_{\ell})$ is also unipotent for $i \geq n+1$.

Step 2.
We show that there exists an open subgroup $I$ of $I_K$ such that the action on $H^n(X_{\bar{K}}, \mathbb{Q}_{\ell})$ is unipotent
and that the index divides $C_{b,c,n}$.

By the previous step it follows that for $i \neq n$ and $s \in J$,
\begin{displaymath}
{\rm Tr}(s; H^i(X_{\bar{K}}, \mathbb{Q}_\ell)) = b_i
\end{displaymath}
and this trace is independent of $\ell$.

By a theorem of Ochiai \cite[Theorem 3.1]{Ochiai:1999}, the alternating sum 
\begin{displaymath}
\sum_{0 \leq i \leq 2n}(-1)^i{\rm Tr}(s; H^i(X_{\bar{K}}, \mathbb{Q}_{\ell}))
\end{displaymath}
is independent of $\ell$ for $s\in I_K$.
Hence the trace ${\rm Tr}(s; H^n(X_{\bar{K}}, \mathbb{Q}_{\ell}))$ is independent of $\ell$ for $s \in J$.

We take a finite extension $L$ over $K$ such that the inertia group $I_L$ equals to $J$.
By applying Corollary~\ref{cor1} to $\rho|_{G_L}$,
there exists an open subgroup $I \subset J$ such that the action of $I$ is unipotent on $H^n(X_{\bar{K}}, \mathbb{Q}_{\ell})$ and $[J:I]|C_{b_n}$.
Also the action on $H^i(X_{\bar{K}}, \mathbb{Q}_{\ell})$ $i\neq n$ is unipotent
and the index $[I_K:I]$ divides $C_{b,c,n-1}C_{b_n} = C_{b,c,n}$.
\end{proof}

We give a few remarks on this theorem.

In \cite[Theorem 3]{Katz:2001},
Katz gives an explicit upper bound for the sum of the Betti number of a projective variety depending on its degree.
The method of the proof is similar to the one given here.

If $K$ is a finite extension of $\mathbb{Q}_p$, then the action of this $I$ on $H^i(X_{\bar{K}}, \mathbb{Q}_p)$ is potentially semi-stable.
In his paper \cite{Ochiai:1999},
Ochiai shows equality of the alternating sum of traces of $H^i(X_{\bar{K}}, \mathbb{Q}_{\ell})$ and $H^i(X_{\bar{K}}, \mathbb{Q}_p)$.
If $K$ is a number field, we can show that there exists an uniform bound for the index in similar way.
That is, there exists a finite extension $L$ of $K$ of bounded degree such that the action is unipotent at $v \nmid\ell$ and semi-stable at $v\mid\ell$.

It is natural to ask whether	 this bound is strict.
By the Poincare duality, we can use the order of SL or Sp instead of GL.
Then, the constant used in this theorem decreases and we will get smaller bounds.

If $X$ is an abelian surface, Silverberg and Zarhin \cite{Silverberg_Zarhin:2005} determine all the groups in the similar way that we discussed in the end of section 2.
They show that these group appear as the inertia group of an abelian surface.

If we know the $\ell$-independence of the trace, we need only the compatible system case.

\end{document}